\documentclass{article}

\usepackage{amsmath,amssymb,amsthm,graphicx,proof,color}
\usepackage[UKenglish]{babel}

\usepackage{hyperref}
\usepackage[all]{xy}
\usepackage{times}
\usepackage{mdframed}

\usepackage{mathtools}

\hypersetup{
   colorlinks,%
   citecolor=blue,%
   filecolor=black,%
   linkcolor=red,%
   urlcolor=black
 }

\usepackage[all]{xy}
\usepackage{tikz}
\usetikzlibrary{trees}
\usepgflibrary{arrows}
\usetikzlibrary{positioning,shapes}
\usetikzlibrary{patterns}

\newcommand{\M}{\mathcal{M}}

\newcommand{\blue}[1]{\textcolor{blue}{#1}}

\newcommand{\BP}{\ensuremath{\textbf{P}}}

\newcommand{\lr}[1]{\langle #1 \rangle}
\newcommand{\lra}{\leftrightarrow}

\renewcommand{\phi}{\varphi}

\newtheorem{theorem}{Theorem}
\newtheorem{lemma}[theorem]{Lemma}
\newtheorem{proposition}[theorem]{Proposition}

\newtheorem{corollary}[theorem]{Corollary}
\newtheorem{remark}[theorem]{Remark}
\newtheorem{fact}[theorem]{Fact}

\theoremstyle{definition}
\newtheorem{definition}[theorem]{Definition}

\newcommand{\weg}[1]{}

\title{A family of neighborhood contingency logics}

\author{Jie Fan\\
\small School of Humanities, University of Chinese Academy of Sciences  \\
\small \texttt{jiefan@ucas.ac.cn}}
\date{}

\begin{document}
\maketitle

\begin{abstract}
This article proposes the axiomatizations of contingency logics of various natural classes of neighborhood frames. In particular, by defining a suitable canonical neighborhood function, we give sound and complete axiomatizations of monotone contingency logic and regular contingency logic, thereby answering two open questions raised by Bakhtiari, van Ditmarsch, and Hansen. The canonical function is inspired by a function proposed by Kuhn in~1995. We show that Kuhn's function is actually equal to a related function originally given by Humberstone.
\end{abstract}

\noindent Keywords: contingency logic, neighborhood semantics, axiomatization, monotone logic, regular logic

\section{Introduction}

Compared to standard modal logic, non-standard modal logics usually have many disadvantages, such as weak expressivity, weak frame definability, which brings about non-triviality of axiomatizations. Contingency logic is such a logic~\cite{MR66,Cresswell88,Humberstone95,DBLP:journals/ndjfl/Kuhn95,DBLP:journals/ndjfl/Zolin99,hoeketal:2004,steinsvold:2008,Fanetal:2014,Fanetal:2015}. \weg{As a non-normal modal logic, c}Contingency logic is concerned with the study of principles of reasoning involving contingency, noncontingency, and related notions. Since it was introduced, contingency logic has mainly been investigated within the framework of Kripke semantics. However, a known pain for axiomatizing this logic over various Kripke frames is the absence of axioms characterizing frame properties. Moreover, although there have been many results on the axiomatizations of contingency logic which are extensions of minimal logic ${\bf K^\Delta}$, there have been no yet much axiomatizations weaker than ${\bf K^\Delta}$, for which we need neighborhood semantics. Since it was independently proposed by Scott and Montague in 1970~\cite{Scott:1970,Montague:1970}, neighborhood semantics has been a standard semantical tool for handling non-normal modal logics~\cite{Chellas:1980}.

A neighborhood semantics of contingency logic is proposed in~\cite{FanvD:neighborhood}. According to the interpretation, a formula $\phi$ is noncontingent, if and only if the proposition expressed by $\phi$ is a neighborhood of the evaluated state, or the complement of the proposition expressed by $\phi$ is a neighborhood of the evaluated state. This interpretation is in line with the philosophical intuition of noncontingency, viz. necessarily true or necessarily false. It is shown that contingency logic is less expressive than standard modal logic over various neighborhood model classes, and many neighborhood frame properties are undefinable in contingency logic. This brings about the difficulties in axiomatizing this logic over various neighborhood frames.

To our knowledge, only the classical contingency logic, i.e. the minimal system of contingency logic under neighborhood semantics, is presented in the literature~\cite{FanvD:neighborhood}. It is left as two open questions in~\cite{Bakhtiarietal:2017} what the axiomatizations of monotone contingency
logic and regular contingency logic are. In this paper, we will answer these two questions.

Besides, we also propose other proof systems up to the minimal Kripke contingency logic, and show their completeness with respect to the corresponding neighborhood frames. This will give a complete diagram which includes 8 systems, as~\cite[Fig.~8.1]{Chellas:1980} did for standard modal logic. It is a sub-diagram of a larger diagram of 16 logics, due to the introduction of a property of being closed under complements.


The remainder of the paper is structured as follows. Section~\ref{sec.preliminaries} introduces some basics of contingency logic, such as its language, neighborhood semantics, axioms and rules. Sections~\ref{sec.sys-no-m} and~\ref{sec.sys-m} deal with the completeness of proof systems mentioned in Sec.~\ref{sec.preliminaries}, with or without a special axiom. The completeness proofs rely on the use of canonical neighborhood functions. In Sec.~\ref{sec.sys-no-m}, a simple canonical function is needed, while in Sec.~\ref{sec.sys-m} we need a more complex canonical function, which is inspired by a crucial function $\lambda$ used in a Kripke completeness proof in the literature. We further reflect on this $\lambda$ in Section~\ref{sec.reflection-lambda}, and show it is in fact equal to a related but complicated function originally given by Humberstone. We conclude with some discussions in Section~\ref{sec.concl}.

\section{Preliminaries}\label{sec.preliminaries}

Throughout this paper, we fix $\BP$ to be a nonempty set of propositional variables. The language $\mathcal{L}(\Delta)$ of contingency logic is defined recursively as follows:
$$\phi::=p\in\BP\mid \neg\phi\mid \phi\land\phi\mid \Delta\phi.$$
$\Delta\phi$ is read ``it is noncontingent that $\phi$''. The contingency operator $\nabla$ abbreviates $\neg\Delta$. It does not matter which one of $\Delta$ and $\nabla$ is taken as primitive. We use $\phi\in\mathcal{L}(\Delta)$ to mean that $\phi$ is an $\mathcal{L}(\Delta)$-formula, and we always leave out the reference to $\mathcal{L}(\Delta)$ and simply say that $\phi$ is a formula.

The neighborhood semantics of $\mathcal{L}(\Delta)$ is interpreted on neighborhood models. We say that $\M=\lr{S,N,V}$ is a {\em neighborhood model} if $S$ is a nonempty set of states, $N:S\to\mathcal{P}(\mathcal{P}(S))$ is a neighborhood function, and $V$ is a valuation assigning a set $V(p)\subseteq S$ to each propositional variable $p$. A {\em neighborhood frame} is a neighborhood model without valuations.

Given a neighborhood model $\M=\lr{S,N,V}$ and a state $s\in S$, the semantics of $\phi\in\mathcal{L}(\Delta)$ is defined recursively as follows~\cite{FanvD:neighborhood},
\[
\begin{array}{|lll|}
\hline
\M,s\vDash p&\iff & s\in V(p)\\
\M,s\vDash \neg\phi&\iff &\M,s\nvDash \phi\\
\M,s\vDash\phi\land\psi&\iff &\M,s\vDash\phi\text{ and }\M,s\vDash\psi\\
\M,s\vDash\Delta\phi&\iff &\phi^\M\in N(s)\text{ or }S\backslash\phi^\M\in N(s)\\
\hline
\end{array}
\]
where $\phi^\M=\{s\in S\mid \M,s\vDash\phi\}$ is the truth set of $\phi$ (i.e. the proposition expressed by $\phi$) in $\M$. Formula $\phi$ is valid in a frame $\mathcal{F}$, notation: $\mathcal{F}\vDash\phi$, if for all models $\M$ based on $\mathcal{F}$ and all $s$ in $\M$, we have that $\M,s\vDash\phi$; $\phi$ is valid on a class $\mathbb{K}$ of frames, notation: $\mathbb{K}\vDash\phi$, if for all $\mathcal{F}$ in $\mathbb{K}$, we have that $\mathcal{F}\vDash\phi$. Notions of validity of a set of formulas in a frame and on a class of frames are defined similarly. Moreover, given a class $\mathbb{K}$ of frames, we say $\mathbb{K}$ is {\em definable} in $\mathcal{L}(\Delta)$, if there is a $\Gamma\subseteq \mathcal{L}(\Delta)$ such that $\mathcal{F}\vDash\Gamma$ iff $\mathcal{F}\in\mathbb{K}$.

\begin{definition}[Neighborhood frame properties]\label{def.properties} Let $\mathcal{F}=\lr{S,N}$ be a neighborhood frame. For every $s\in S$ and every $X,Y\subseteq S$:

$(m)$: $N(s)$ is \emph{supplemented}, or \emph{closed under supersets}, if $X\in N(s)$ and $X\subseteq Y\subseteq S$ implies $Y\in N(s)$.

$(c)$: $N(s)$ \emph{is closed under intersections}, if $X,Y\in N(s)$ implies $X\cap Y\in N(s)$.

$(n)$: $N(s)$ \emph{contains the unit}, if $S\in N(s)$. 



$(z)$: $N(s)$ is\emph{ closed under complements}, if $X\in N(s)$ implies $S\backslash X\in N(s)$.\footnote{The property $(z)$ was introduced in~\cite[Def.~3]{FanvD:neighborhood}, named `$(c)$' therein.} 
\end{definition}
Frame $\mathcal{F}=\lr{S,N}$ (and the corresponding model) possesses such a property P, if $N(s)$ has the property P for each $s\in S$, and we call the frame (resp. the model) P-frame (resp. P-model). Especially, a frame is called {\em quasi-filter}, if it possesses $(m)$ and $(c)$; a frame is called {\em filter}, if it has also $(n)$. The property $(z)$ is needed for the following soundness and completeness results, and it provides us a new perspective (see~\cite{Fan:2018a}) for the neighborhood semantics of $\mathcal{L}(\Delta)$. All properties listed above are shown to be undefinable in $\mathcal{L}(\Delta)$~\cite[Prop.~7]{FanvD:neighborhood}. In contrast, they are definable in standard modal logic $\mathcal{L}(\Box)$.\footnote{$\mathcal{L}(\Box)$ extends the language of propositional logic with the necessity operator $\Box$, formally defined as follows: $$\phi::=p\in\BP\mid \neg\phi\mid\phi\land\phi\mid \Box\phi,$$
where the neighborhood semantics of $\Box$ is
$$\begin{array}{lll}
\M,s\vDash\Box\phi&\iff&\phi^\M\in N(s).
\end{array}
$$} 
The proofs of the first three can be found in~\cite[Thm.~7.5, Thm.~9.2]{Chellas:1980}, and the proof of the last one is similar to~\cite[Prop.~5]{Fan:2018a}.
\begin{fact}\label{fact.correspondence-box}
The frame properties on the left are respectively defined by the formulas on the right:
\[\begin{array}{llllll}
(m)&\Box(p\land q)\to(\Box p\land\Box q)&&&(c)&(\Box p\land\Box q)\to\Box(p\land q)\\
(n)&\Box\top&&&(z)&\Box p\to\Box\neg p\\
\end{array}\]
\end{fact}

Recall the axioms and rules in 8 classical modal systems and the classes of frames determining them listed below, see e.g.~\cite[Chap.~8]{Chellas:1980}.
\[
\begin{array}{lllll}
\text{TAUT}&\text{all instances of propositional tautologies}&&&\\
\text{M}&\Box(\phi\land\psi)\to(\Box\phi\land\Box\psi)&&\text{MP}&\dfrac{\phi,\phi\to\psi}{\psi}\\
\text{C}&(\Box\phi\land\Box\psi)\to\Box(\phi\land\psi)&&\text{RE}&\dfrac{\phi\lra\psi}{\Box\phi\lra\Box\psi}\\
\text{N}&\Box\top&&&\\
\end{array}
\]
\[
\begin{array}{|l|l|}
\hline
\text{systems} & \text{frame classes} \\
\hline
{\bf E}=\text{TAUT}+\text{MP}+\text{RE}&\text{all}\\
{\bf M}(={\bf EM})={\bf E}+\text{M}&(m)\\
{\bf EC}={\bf E}+\text{C}&(c)\\
{\bf EN}={\bf E}+\text{N}&(n)\\
{\bf R}(={\bf EMC})={\bf M}+\text{C}&\text{quasi-filters}=(mc)\\
{\bf EMN}={\bf M}+\text{N}&(mn)\\
{\bf ECN}={\bf EC}+\text{N}&(cn)\\
{\bf K}(={\bf EMCN})={\bf R}+\text{N}&\text{filters}=(mcn)\\
\hline
\end{array}
\]
\weg{Our discussions will be based on the following axioms and rules (for comparison's sake, we also introduce related axioms and rules based on $\mathcal{L}(\Box)$, which are on the right).
\[
\begin{array}{llll}
\text{TAUT}& \text{all instances of propositional tautologies}&\text{TAUT}&\text{all instances of propositional tautologies}\\
\Delta\text{Equ}&\Delta\phi\to\Delta\neg\phi&\text{Z}&\Box\phi\to\Box\neg\phi\\
\Delta\text{M}&\Delta\phi\to\Delta(\phi\vee\psi)\vee\Delta(\neg\phi\vee\chi)&\text{M}&\Box(\phi\land\psi)\to(\Box\phi\land\Box\psi)\\
\Delta\text{C}&\Delta\phi\land\Delta\psi\to\Delta(\phi\land\psi)&\text{C}&(\Box\phi\land\Box\psi)\to\Box(\phi\land\psi)\\
\Delta\text{N}&\Delta\top&\text{N}&\Box\top\\
\text{MP}&\dfrac{\phi,\phi\to\psi}{\psi}&\text{MP}&\dfrac{\phi,\phi\to\psi}{\psi}\\
\text{RE}\Delta&\dfrac{\phi\lra\psi}{\Delta\phi\lra\Delta\psi}&\text{RE}&\dfrac{\phi\lra\psi}{\Box\phi\lra\Box\psi}\\
\end{array}
\]}

Our discussions will mainly be based on the following axioms and rules.
\[
\begin{array}{lllll}
\text{TAUT}& \text{all instances of propositional tautologies}&&\Delta\text{N}&\Delta\top\\
\Delta\text{Equ}&\Delta\phi\lra\Delta\neg\phi&&\text{MP}&\dfrac{\phi,\phi\to\psi}{\psi}\\
\Delta\text{M}&\Delta\phi\to\Delta(\phi\vee\psi)\vee\Delta(\neg\phi\vee\chi)&&\text{RE}\Delta&\dfrac{\phi\lra\psi}{\Delta\phi\lra\Delta\psi}\\
\Delta\text{C}&\Delta\phi\land\Delta\psi\to\Delta(\phi\land\psi)&&&\\
\end{array}
\]

\weg{\[
\begin{array}{ll}
\text{TAUT}& \text{all instances of propositional tautologies}\\
\Delta\text{Equ}&\Delta\phi\lra\Delta\neg\phi\\
\Delta\text{M}&\Delta\phi\to\Delta(\phi\vee\psi)\vee\Delta(\neg\phi\vee\chi)\\
\Delta\text{C}&\Delta\phi\land\Delta\psi\to\Delta(\phi\land\psi)\\
\Delta\text{N}&\Delta\top\\
\text{MP}&\dfrac{\phi,\phi\to\psi}{\psi}\\
\text{RE}\Delta&\dfrac{\phi\lra\psi}{\Delta\phi\lra\Delta\psi}\\
\end{array}
\]}

We will show that the following systems are sound and strongly complete with respect to their corresponding frame classes.
\[
\begin{array}{|l|l|}
\hline
\text{systems} & \text{frame classes} \\
\hline
{\bf E^\Delta}={\bf EZ^\Delta}=\text{TAUT}+\Delta\text{Equ}+\text{MP}+\text{RE}\Delta&\text{all }(\text{also }(z))\\
{\bf M^\Delta}(={\bf EM^\Delta})={\bf E^\Delta}+\Delta\text{M}&(m)\\
{\bf ECZ^\Delta}={\bf E^\Delta}+\Delta\text{C}&(cz)\\
{\bf EN^\Delta}={\bf ENZ^\Delta}={\bf E^\Delta}+\Delta\text{N}&(n) (\text{also }(nz))\\
{\bf R^\Delta}(={\bf EMC^\Delta})={\bf M^\Delta}+\Delta\text{C}&\text{quasi-filters}=(mc)\\
{\bf EMN^\Delta}={\bf M^\Delta}+\Delta\text{N}&(mn)\\
{\bf ECNZ^\Delta}={\bf ECZ^\Delta}+\Delta\text{N}&(cnz)\\
{\bf K^\Delta}(={\bf EMCN^\Delta})={\bf R^\Delta}+\Delta\text{N}&\text{filters}=(mcn)\\
\hline
\end{array}
\]

The notion of theorems in a system is defined as normal. 

By comparison, one can easily see that almost all of $\Delta$-systems and the corresponding $\Box$-systems mentioned above are determined by the same class of frames, but with two exceptions: even though ${\bf EC}$ and ${\bf ECN}$ are respectively determined by the class of $(c)$-frames and the class of $(cn)$-frames, we have only $\Delta$-systems which are respectively determined by the class of $(cz)$-frames and the class of $(cnz)$-frames, that is --- ${\bf ECZ^\Delta}$ and ${\bf ECNZ^\Delta}$. We do not know whether there are axiomatizations of contingency logics over $(c)$-frames and over $(cn)$-frames.

Given a $\Box$-logic ${\bf L}$ (that is, logic in $\mathcal{L}(\Box)$), one can define a $\Delta$-logic, denoted ${\bf L^\Delta}$, as ${\bf L^\Delta}=\{\phi\in \mathcal{L}(\Delta)\mid \phi^\star\in{\bf L}\}$, where $\phi^\star$ is defined inductively, with $(\Delta\phi)^\star=\Box\phi^\star\vee\Box\neg\phi^\star$. In other words, ${\bf L^\Delta}$ proves exactly those $\mathcal{L}(\Delta)$-formulas whose $^\star$-translations are provable in ${\bf L}$. It is easy to show that if ${\bf L}$ is the $\Box$-logic of some class of frames $\mathbb{K}$ (in symbol, ${\bf L}=\text{Th}_\Box(\mathbb{K})$), that is, ${\bf L}$ is the set of $\mathcal{L}(\Box)$-formulas that are valid in $\mathbb{K}$, then ${\bf L^\Delta}$ is the $\Delta$-logic of $\mathbb{K}$ (in symbol, ${\bf L^\Delta}=\text{Th}_\Delta(\mathbb{K})$). Note that one cannot obtain the axiomatization of ${\bf L^\Delta}$ from the axiomatization of ${\bf L}$, since there is no translation function from $\mathcal{L}(\Box)$ to $\mathcal{L}(\Delta)$. 

Recall that Def.~\ref{def.properties} listed 4 frame properties, which constitutes $16=2^4$ different combinations of such properties. Since every frame class $\mathbb{K}$ can define a $\Delta$-logic, namely $\text{Th}_\Delta(\mathbb{K})$, we should have 16 different $\Delta$-logics. In this paper, we axiomatize 10 $\Delta$-logics as listed above, and leave the axiomatizations of the remaining 6 $\Delta$-logics open. We will defer a summary of 16 logics with some remarks to the end of Sec.~\ref{sec.sys-m}.

In what follows, we also use $\text{s}\Delta\text{M}$ to denote $\Delta\phi\to\Delta(\phi\vee\psi)$, which is clearly `stronger' than $\Delta\text{M}$, $\Delta\phi\to\Delta(\phi\vee\psi)\lor \Delta(\neg\phi\vee\chi)$. Note that $\text{s}\Delta\text{M}$ is equivalent to $\Delta(\phi\land\psi)\to\Delta\phi$ in ${\bf E^\Delta}$, since they are interderivable with the rule of $\Delta$-monotony $\dfrac{\phi\to\psi}{\Delta\phi\to\Delta\psi}$ in the system in question.

\medskip

Let $\mathbb{F}_{\text{all}}$ denote the class of all frames, $\mathbb{F}_m$ denote the class of $(m)$-frames, $\mathbb{F}_c$ denote the class of $(c)$-frames, and similarly for other properties. 
\begin{proposition}\label{prop.validinvalid} We have the following validities and invalidities:
\begin{itemize}
\item[(i)] $\mathbb{F}_n\vDash\Delta\text{N}$.
\item[(ii)] $\mathbb{F}_m\vDash\Delta\text{M}$.
\item[(iii)] $\mathbb{F}_{cz}\vDash\Delta\text{C}$.
\item[(iv)] $\mathbb{F}_{mc}\vDash\Delta\text{C}$.\footnote{It is worth remarking that in the case of $(c)$-frames, we need the property $(z)$ to provide the validity of $\Delta$C (see the proof of item (iii) in this proposition); by comparison, in the case of quasi-filters, we do not need $(z)$, since the validity of $\Delta$C is now guaranteed by $(c)$ and $(m)$ together.}
\item[(v)] $\mathbb{F}_c\nvDash\Delta\text{C}$.
\item[(vi)] $\mathbb{F}_{cn}\nvDash\Delta\text{C}$.
\item[(vii)] $\mathbb{F}_{mcn}\nvDash\text{s}\Delta\text{M}$.
\item[(viii)] $\mathbb{F}_{m}\nvDash\text{s}\Delta\text{M}$.
\item[(ix)] $\mathbb{F}_{mc}\nvDash\text{s}\Delta\text{M}$.
\item[(x)] $\mathbb{F}_{mn}\nvDash\text{s}\Delta\text{M}$.
\item[(xi)] $\mathbb{F}_{mz}\vDash\text{s}\Delta\text{M}$.
\end{itemize}
\end{proposition}

\begin{proof}\
\begin{itemize}
\item[(i)] Let $\M=\lr{S,N,V}$ be an $(n)$-model and $s\in S$. By $(n)$, $S\in N(s)$, that is, $\top^\M\in N(s)$, and thus $\top^\M\in N(s)$ or $S\backslash \top^\M\in N(s)$, and hence $\M,s\vDash\Delta\top$. By the arbitrariness of $\M$ and $s$, we conclude that $\mathbb{F}_n\vDash\Delta\text{N}$.
\item[(ii)] Let $\M=\lr{S,N,V}$ be an $(m)$-model and $s\in S$. Suppose that $\M,s\vDash\Delta\phi$, then $\phi^\M\in N(s)$ or $(\neg\phi)^\M\in N(s)$. If $\phi^\M\in N(s)$, then by $(m)$, $\phi^\M\cup\psi^\M\in N(s)$, which implies $\M,s\vDash\Delta(\phi\vee\psi)$; if $(\neg\phi)^\M\in N(s)$, then similarly, we can obtain $\M,s\vDash\Delta (\neg\phi\vee\chi)$. Either case gives us $\M,s\vDash\Delta (\phi\vee\psi)\vee\Delta(\neg\phi\vee\chi)$, as required.
\item[(iii)] Let $\mathcal{M}=\lr{S,N,V}$ be a $(cz)$-model and $s\in S$. Suppose $\M,s\vDash\Delta\phi$ and $\M,s\vDash\Delta\psi$, to show $\M,s\vDash\Delta(\phi\land\psi)$. From $\M,s\vDash\Delta\phi$ it follows that $\phi^M\in N(s)$ or $S\backslash \phi^\M\in N(s)$. Using $(z)$, we can infer $\phi^\M\in N(s)$. Similarly, from $\M,s\vDash\Delta\psi$ we can obtain $\psi^\M\in N(s)$. Now an application of $(c)$ gives us $\phi^\M\cap \psi^\M\in N(s)$, that is, $(\phi\land\psi)^\M\in N(s)$, and thus $\M,s\vDash\Delta(\phi\land\psi)$.
\item[(iv)] Let $\M=\lr{S,N,V}$ be a quasi-filter model and $s\in S$. Suppose that $\M,s\vDash\Delta\phi\land\Delta\psi$, then $\phi^\M\in N(s)$ or $(\neg\phi)^\M\in N(s)$, and $\psi^\M\in N(s)$ or $(\neg\psi)^\M\in N(s)$. Consider the following three cases:
\begin{itemize}
\item $\phi^\M\in N(s)$ and $\psi^\M\in N(s)$. By $(c)$, we obtain $\phi^\M\cap \psi^\M\in N(s)$, i.e. $(\phi\land\psi)^\M\in N(s)$, which gives $\M,s\vDash\Delta(\phi\land\psi)$.
\item $(\neg\phi)^\M\in N(s)$. By $(m)$, we infer $(\neg\phi)^\M\cup(\neg\psi)^\M\in N(s)$, i.e. $(\neg(\phi\land\psi))^\M\in N(s)$, which implies $\M,s\vDash\Delta(\phi\land\psi)$.
\item $(\neg\psi)^\M\in N(s)$. Similar to the second case, we can derive that $\M,s\vDash\Delta(\phi\land\psi)$.
\end{itemize}
\item[(v)] Consider an instance of $\Delta$C: $\Delta p\land\Delta\neg p\to\Delta (p\land \neg p)$ and a model $\M=\lr{S,N,V}$ where $S=\{s,t\}$, $N(s)=\{\{s\}\}$, $N(t)=\emptyset$, and $V(p)=\{s\}$. It should be obvious that $\M$ is a $(c)$-model. On the other hand, since $p^\M=\{s\}\in N(s)$, we have $\M,s\vDash\Delta p$, also, as $S\backslash (\neg p)^\M=p^\M\in N(s)$, we infer $\M,s\vDash\Delta\neg p$; however, $(p\land \neg p)^\M=\emptyset\notin N(s)$ and $S\backslash (p\land \neg p)^\M=S\notin N(s)$, thus $\M,s\nvDash\Delta(p\land\neg p)$. Therefore, $\M,s\nvDash\Delta p\land\Delta\neg p\to\Delta (p\land \neg p)$. We have thus found a $(c)$-model which falsifies an instance of $\Delta$C, and it can be concluded that $\mathbb{F}_c\nvDash\Delta\text{C}$.
\item[(vi)] Consider an instance of $\Delta$C: $\Delta p\land\Delta q\to\Delta (p\land q)$ and a model $\M=\lr{S,N,V}$ where $S=\{s,t,u\}$, $N(s)=\{\{s\},\{s,t\},S\}$, $N(t)=N(u)=\{S\}$, and $V(p)=\{s,t\}$, $V(q)=\{t,u\}$. It should be clear that $\M$ is a $(cn)$-model. Since $p^\M=\{s,t\}\in N(s)$, we have $\M,s\vDash\Delta p$; since $S\backslash q^\M=\{s\}\in N(s)$, thus $\M,s\vDash\Delta q$. However, $(p\land q)^\M=\{t\}\notin N(s)$ and $S\backslash (p\land q)^\M=\{s,u\}\notin N(s)$, thus $\M,s\nvDash\Delta(p\land q)$. Therefore, $\M,s\nvDash\Delta p\land\Delta q\to\Delta (p\land q)$. We have thus found a $(cn)$-model which falsifies an instance of $\Delta$C, and it can be concluded that $\mathbb{F}_{cn}\nvDash\Delta\text{C}$.
\item[(vii)] Consider the following instance of s$\Delta$M: $\Delta p\to\Delta(p\vee q)$, and a model $\M=\lr{S,N,V}$ where $S=\{s,t,u\}$, $N(s)=\{\{t,u\},S\}$, $N(t)=N(u)=\{S\}$, and $V(p)=\{s\}$, $V(q)=\{t\}$. One may easily verify that $\M$ is an $(mcn)$-model. However, $\M,s\nvDash\Delta p\to\Delta(p\vee q)$: on one hand, as $(\neg p)^\M=\{t,u\}\in N(s)$, we have $\M,s\vDash\Delta p$; on the other hand, since $(p\vee q)^\M=\{s,t\}\notin N(s)$ and $(\neg (p\vee q))^\M=\{u\}\notin N(s)$, thus $\M,s\nvDash\Delta (p\vee q)$.
\item[(viii)-(x)] Follows directly from item (vii), since $(mcn)$-models are also $(m)$-models, $(mc)$-models and $(mn)$-models.
\item[(xi)] Let $\M=\lr{S,N,V}$ be an $(mz)$-model and $s\in S$. Suppose that $\M,s\vDash\Delta \phi$, to show that $\M,s\vDash\Delta(\phi\vee\psi)$. By supposition, we have $\phi^\M\in N(s)$ or $(\neg\phi)^\M\in N(s)$. By $(z)$, it follows that $\phi^\M\in N(s)$. Then by $(m)$, it follows that $\phi^\M\cup \psi^\M\in N(s)$, viz. $(\phi\vee\psi)^\M\in N(s)$, and therefore $\M,s\vDash\Delta(\phi\vee\psi)$.
\end{itemize}
\end{proof}

\begin{corollary}[Soundness] The aforementioned 10 logics are sound with respect to their corresponding class of frames.
\begin{itemize}
\item ${\bf E^\Delta}={\bf EZ^\Delta}\subseteq \text{Th}_\Delta(\mathbb{F}_{\text{all}})\subseteq \text{Th}_\Delta(\mathbb{F}_z)$
\item ${\bf M^\Delta}\subseteq \text{Th}_\Delta(\mathbb{F}_m)$
\item ${\bf ECZ^\Delta}\subseteq \text{Th}_\Delta(\mathbb{F}_{cz})$
\item ${\bf EN^\Delta}={\bf ENZ^\Delta}\subseteq \text{Th}_\Delta(\mathbb{F}_n)\subseteq \text{Th}_\Delta(\mathbb{F}_{nz})$
\item ${\bf R^\Delta}\subseteq \text{Th}_\Delta(\mathbb{F}_{mc})$
\item ${\bf EMN^\Delta}\subseteq \text{Th}_\Delta(\mathbb{F}_{mn})$
\item ${\bf ECNZ^\Delta}\subseteq \text{Th}_\Delta(\mathbb{F}_{cnz})$
\item ${\bf K^\Delta}\subseteq \text{Th}_\Delta(\mathbb{F}_{mcn})$
\end{itemize}
\end{corollary}

From the next section, we will start to show the completeness results of these systems, with the aid of canonical neighborhood model constructions. As one will see, all the above systems may not be handled by a uniform canonical neighborhood function; instead, we need to distinguish systems excluding axiom $\Delta$M from those including it.

Given a system $\Lambda$ and the set $S^\Lambda$ of all maximal consistent sets for $\Lambda$,\weg{\footnote{The notation of superscript $(\--)^c$ should not be confused with the set complement notation $\backslash$, which has been seen when we define property $(z)$ in Def.~\ref{def.properties}.}} let $|\phi|_\Lambda$ be the {\em proof set} of $\phi$ relative to $\Lambda$, in symbol, $|\phi|_\Lambda=\{s\in S^\Lambda\mid \phi\in s\}$.\footnote{The terminology `the proof set' can be found on ~\cite[p. 57]{Chellas:1980}.} It is easy to show that $|\neg\phi|_\Lambda=S^\Lambda\backslash |\phi|_\Lambda$ and $|\phi\land\psi|_\Lambda=|\phi|_\Lambda\cap |\psi|_\Lambda$. We always omit the subscript $\Lambda$ when it is clear from the context.

\section{Systems excluding $\Delta$M}\label{sec.sys-no-m}

Given a proof system, a standard method of showing its completeness under neighborhood semantics is constructing the canonical neighborhood model, where one essential part is the definition of canonical neighborhood function.


\begin{definition}\label{def.cm-no-m} Let $\Lambda$ be a system excluding $\Delta$M.
A tuple $\M^\Lambda=\lr{S^\Lambda,N^\Lambda,V^\Lambda}$ is the {\em canonical neighborhood model} for $\Lambda$, if
\begin{itemize}
\item $S^\Lambda=\{s\mid s\text{ is a maximal consistent set for }\Lambda\}$,
\item $N^\Lambda(s)=\{|\phi|\mid \Delta\phi\in s\}$,
\item $V^\Lambda(p)=|p|$.
\end{itemize}
\end{definition}

Notice that thanks to axiom $\Delta$Equ, the function $N^\Lambda$ in the above definition has the property $(z)$, that is, for all $s\in S^\Lambda$ and $X\subseteq S^\Lambda$, if $X\in N^\Lambda(s)$, then $S^\Lambda\backslash X\in N^\Lambda(s)$.

\begin{theorem}\cite[Thm.~1, Thm.~2]{FanvD:neighborhood}\label{thm.comp-E}
${\bf E^\Delta}={\bf EZ^\Delta}$ is strongly complete with respect to the class of all neighborhood frames and also w.r.t. the class of all $(z)$-frames. Therefore, ${\bf E^\Delta}=\text{Th}_\Delta(\mathbb{F}_\text{all})={\bf EZ^\Delta}=\text{Th}_\Delta(\mathbb{F}_z)$.
\end{theorem}

In what follows, we will extend the canonical model construction to all systems excluding $\Delta$M listed above.



\begin{theorem}\label{thm.comp-EC}
${\bf ECZ^\Delta}$ is strongly complete with respect to the class of all $(cz)$-frames.
\end{theorem}

\begin{proof}
By Thm.~\ref{thm.comp-E}, it suffices to show that $N^\Lambda$ possesses $(c)$. This is guaranteed by axiom $\Delta$C: suppose $X\in N^\Lambda(s)$ and $Y\in N^\Lambda(s)$, then by definition of $N^\Lambda$, $X=|\phi|\in N^\Lambda(s)$ and $Y=|\psi|\in N^\Lambda(s)$ for some $\phi$ and $\psi$, thus $\Delta\phi\in s$ and $\Delta\psi\in s$, which implies $\Delta(\phi\land\psi)\in s$ because of axiom $\Delta$C, and therefore $|\phi\land \psi|\in N^\Lambda(s)$, that is, $X\cap Y\in N^\Lambda(s)$.
\weg{By Thm.~\ref{thm.comp-E}, it suffices to show that $\Delta$C is valid on $(cz)$-frames, and that $N^c$ possesses $(i)$ and $(c)$.

Given any $(i)\&(c)$-model $\mathcal{F}=\lr{S,N,V}$ and $s\in S$, suppose $\M,s\vDash\Delta\phi$ and $\M,s\vDash\Delta\psi$, to show $\M,s\vDash\Delta(\phi\land\psi)$. From $\M,s\vDash\Delta\phi$ it follows that $\phi^M\in N(s)$ or $S\backslash \phi^\M\in N(s)$. Using $(c)$, we can infer $\phi^\M\in N(s)$. Similarly, from $\M,s\vDash\Delta\psi$ we can obtain $\psi^\M\in N(s)$. Now an application of $(i)$ gives us $\phi^\M\cap \psi^\M\in N(s)$, that is, $(\phi\land\psi)^\M\in N(s)$, and thus $\M,s\vDash\Delta(\phi\land\psi)$.

As for the latter, $\Delta$Equ guarantees $(c)$, and $\Delta$C provides $(i)$.}
\end{proof}

\begin{theorem}\label{thm.comp-EN}
${\bf EN^\Delta}={\bf ENZ^\Delta}$ is strongly complete with respect to the class of all $(n)$-frames and also w.r.t. the class of all $(nz)$-frames. Therefore, ${\bf EN^\Delta}=\text{Th}_\Delta(\mathbb{F}_n)={\bf ENZ^\Delta}=\text{Th}_\Delta(\mathbb{F}_{nz})$.
\end{theorem}

\begin{proof}
By Thm.~\ref{thm.comp-E}, it suffices to show that $N^\Lambda$ possesses the property $(n)$.\weg{\footnote{There may be other canonical neighborhood functions which also work for this theorem, but we choose this function, is mainly because the function is very simple.}} This is immediate due to $\Delta$N and the definition of $N^\Lambda$: since $\vdash\Delta\top$, we have that for all $s\in S^\Lambda$, $\Delta\top\in s$, and then $|\top|\in N^\Lambda(s)$, that is, $S^c\in N^\Lambda(s)$.
\end{proof}

\begin{theorem}
${\bf ECNZ^\Delta}$ is strongly complete with respect to the class of all $(cnz)$-frames.
\end{theorem}

\begin{proof}
By Thm.~\ref{thm.comp-E}, it suffices to show that $N^\Lambda$ has the properties $(c)$ and $(n)$. Property $(c)$ is provided by axiom $\Delta$C (see the proof of Thm.~\ref{thm.comp-EC}), and $(n)$ is warranted by axiom $\Delta$N (see the proof of Thm.~\ref{thm.comp-EN}).
\end{proof}

\section{Systems including $\Delta$M}\label{sec.sys-m}

In this section, we show that the systems including $\Delta$M listed above are strongly complete with respect to the corresponding frame classes. For this, we construct the canonical neighborhood model for any system extending ${\bf M^\Delta}$, where the crucial definition is the canonical neighborhood function. The definition of $N^\Lambda$ below is inspired by a function $\lambda$ introduced in~\cite{DBLP:journals/ndjfl/Kuhn95}.\footnote{The difference between $N^\Lambda$ and $\lambda$ lies in the codomains: $N^\Lambda$'s codomain is $\mathcal{P}(\mathcal{P}(S^\Lambda))$, whereas $\lambda$'s is $\mathcal{P}(\mathcal{L}(\Delta))$.}

\begin{definition}\label{def.canonicalmodel} Let $\Lambda$ be a system extending ${\bf M^\Delta}$. 
A triple $\M^\Lambda=\lr{S^\Lambda,N^\Lambda,V^\Lambda}$ is a {\em canonical neighborhood model} for $\Lambda$, if\weg{\footnote{Note that there may be various canonical neighborhood models, as long as each of them satisfies the conditions specified by this definition. This is in line with the case of classical modal logics, see~\cite[Chap.~9]{Chellas:1980}.}}
\begin{itemize}
\item $S^\Lambda=\{s\mid s\text{ is a maximal consistent set for }\Lambda\}$,
\item For each $s\in S^\Lambda$, $|\phi|\in N^\Lambda(s)$ iff $\Delta(\phi\vee\psi)\in s\text{ for every }\psi$,
\item For each $p\in \BP$, $V^\Lambda(p)=|p|$.
\end{itemize}
\end{definition}

We need to show that $N^\Lambda$ is well-defined.
\begin{lemma} Let $s\in S^\Lambda$ as defined in Def.~\ref{def.canonicalmodel}.
If $|\phi|=|\phi'|$, then $(\Delta(\phi\vee\psi)\in s\text{ for every }\psi)$ iff $(\Delta(\phi'\vee\psi)\in s\text{ for every }\psi)$.
\end{lemma}

\begin{proof}
Suppose that $|\phi|=|\phi'|$, then $\vdash\phi\lra\phi'$, then for every $\psi$, $\vdash\phi\vee\psi\lra\phi'\vee\psi$. By RE$\Delta$, we have $\vdash\Delta(\phi\vee\psi)\lra\Delta(\phi'\vee\psi)$, thus (for every $\psi$, $\Delta(\phi\vee\psi)\in s$) iff (for every $\psi$, $\Delta(\phi'\vee\psi)\in s$).
\end{proof}

Def.~\ref{def.canonicalmodel} does {\em not} specify the function $N^\Lambda$ completely; besides the sets of the form $|\phi|$ that satisfy this definition, $N^\Lambda(s)$ may contain other sets $X\subseteq S^\Lambda$ that are not of the form $|\phi|$ for any $\mathcal{L}(\Delta)$-formula $\phi$. Therefore, each logic under consideration has many canonical models.

\begin{lemma}\label{lemma.truthlemma-mc} Let $\M^\Lambda$ be an arbitrary canonical model for any system extending ${\bf M^\Delta}$\weg{including axiom $\Delta$M}. Then for all $\phi\in \mathcal{L}(\Delta)$, for all $s\in S^\Lambda$, we have
$\M^\Lambda,s\vDash\phi\iff \phi\in s,$ i.e. $\phi^{\M^\Lambda}=|\phi|$.
\end{lemma}

\begin{proof}
By induction on $\phi$. The base case and Boolean cases are straightforward by Def.~\ref{def.canonicalmodel} and induction hypothesis. The only nontrivial case is $\Delta\phi$.


Suppose, for a contradiction, that $\Delta\phi\in s$ but $\M^\Lambda,s\nvDash\Delta\phi$. Then by induction hypothesis, we obtain $|\phi|\notin N^\Lambda(s)$, and $S^\Lambda\backslash|\phi|\notin N^\Lambda(s)$, i.e. $|\neg\phi|\notin N^\Lambda(s)$. Thus $\Delta(\phi\vee\psi)\notin s$ for some $\psi$, and $\Delta(\neg\phi\vee\chi)\notin s$ for some $\chi$. Using axiom $\Delta$M, we obtain $\Delta\phi\notin s$: a contradiction.

Conversely, assume that $\M^\Lambda,s\vDash\Delta\phi$, to show that $\Delta\phi\in s$. By assumption and induction hypothesis, we have $|\phi|\in N^\Lambda(s)$, or $S^\Lambda\backslash|\phi|\in N^\Lambda(s)$, i.e. $|\neg\phi|\in N^\Lambda(s)$. If $|\phi|\in N^\Lambda(s)$, then for every $\psi$, $\Delta(\phi\vee\psi)\in s$. In particular, $\Delta\phi\in s$; if $|\neg\phi|\in N^\Lambda(s)$, then by a similar argument, we obtain $\Delta\neg\phi\in s$, thus $\Delta\phi\in s$. Therefore, $\Delta\phi\in s$.
\end{proof}

\weg{Similar to the case of classical modal logics~\cite[Sec.~9.2]{Chellas:1980}, given a system $\Gamma$ extending ${\bf M^\Delta}$, there are various canonical neighborhood models for $\Gamma$, as long as $N^c(s)$ always consists of the set $\{|\phi|_\Gamma\mid\Delta(\phi\vee\psi)\in s\text{ for every }\psi\}$ together with a collection of non-proof sets relative to $\Gamma$, in symbol,
$$N^c(s)=\{|\phi|_\Gamma\mid \Delta(\phi\vee\psi)\in s\text{ for every }\psi\}\cup \mathfrak{X},$$
where $\mathfrak{X}\subseteq\{X\subseteq S^c\mid X\neq |\phi|_\Gamma\text{ for every }\phi\}$. In fact, similar to~\cite[Thm.~9.7]{Chellas:1980}, we have that
\begin{proposition}
Let $\Gamma$ be a system extending ${\bf M^\Delta}$. Then $\M^c=\lr{S^c,N^c,V^c}$ is a canonical neighborhood model for $\Gamma$ iff $S^c$ and $V^c$ are as in Def.~\ref{def.canonicalmodel}, and for each $s\in S^c$,
$$N^c(s)=\{|\phi|_\Gamma\mid \Delta(\phi\vee\psi)\in s\text{ for every }\psi\}\cup \mathfrak{X},$$
where $\mathfrak{X}$ is defined as above.
\end{proposition}}

Given a system $\Lambda$ extending ${\bf M^\Delta}$, the minimal canonical model for $\Lambda$, denoted $\M_0^\Lambda=\lr{S^\Lambda,N^\Lambda_0,V^\Lambda}$, is defined where $N^\Lambda_0(s)=\{|\phi|\mid \Delta(\phi\vee\psi)\in s\text{ for every }\psi\}$. Note that $\M_0^\Lambda$ is not necessarily supplemented. Thus we need to define a notion of supplementation, which comes from~\cite{Chellas:1980}.

\begin{definition}
Let $\M=\lr{S,N,V}$ be a neighborhood model. The {\em supplementation} of $\M$, denoted $\M^+$, is a triple $\lr{S,N^+,V}$, where for each $s\in S$, $N^+(s)$ is the superset closure of $N(s)$, i.e. for every $s\in S$,
$$N^+(s)=\{X\subseteq S\mid Y\subseteq X\text{ for some }Y\in N(s)\}.$$
Intuitively, $\M^+$ differs from $\M$ only in that $N^+(s)$ contains every proposition in $\M$ that includes any proposition in $N(s)$.
\end{definition}

It is easy to see that $\M^+$ is supplemented. Moreover, $N(s)\subseteq N^+(s)$. The proof below is a routine work.

\begin{proposition}\label{prop.i&n}
Let $\M$ be a neighborhood model. If $\M$ possesses the property $(c)$, then so does $\M^+$; if $\M$ possesses the property $(n)$, then so does $\M^+$.
\end{proposition}

We will denote the supplementation of $\M_0^\Lambda$ by $(\M_0^\Lambda)^+=\lr{S^\Lambda,(N^\Lambda_0)^+,V^\Lambda}$. By definition of $(N^\Lambda_0)^+$, $(\M_0^\Lambda)^+$ is an $(m)$-model. To demonstrate the completeness of any system extending ${\bf M^\Delta}$ with respect to the class of $(m)$-frames, we need only to show that $(\M_{0}^\Lambda)^+$ is a canonical neighborhood model for any system extending ${\bf M^\Delta}$. 
That is,
\begin{lemma}\label{lem.canonicalmodel} For each $s\in S^\Lambda$,
$$|\phi|\in (N^\Lambda_0)^+(s)\iff \Delta(\phi\vee\psi)\in s\text{ for every }\psi.$$
\end{lemma}

\begin{proof}
`$\Longleftarrow$': immediate by $N^\Lambda_0(s)\subseteq (N^\Lambda_0)^+(s)$ for each $s\in S^\Lambda$ and the definition of $N^\Lambda_0$.

`$\Longrightarrow$': Suppose that $|\phi|\in (N^\Lambda_0)^+(s)$, to show that $\Delta(\phi\vee\psi)\in s\text{ for every }\psi$. By supposition, $X\subseteq |\phi|$ for some $X\in N^\Lambda_0(s)$. Then there is a $\chi$ such that $|\chi|=X$, and thus $\Delta(\chi\vee\psi)\in s$ for every $\psi$, in particular $\Delta(\chi\vee\phi\vee\psi)\in s$. From $|\chi|\subseteq |\phi|$ follows that $\vdash\chi\to\phi$, thus $\vdash\chi\vee\phi\vee\psi\lra\phi\vee\psi$, and hence $\vdash\Delta(\chi\vee\phi\vee\psi)\lra\Delta(\phi\vee\psi)$ by RE$\Delta$. Therefore $\Delta(\phi\vee\psi)\in s$ for every $\psi$.
\end{proof}

\weg{\begin{lemma}
For all $\phi\in \mathcal{L}(\Delta)$, for all $s\in S^c$, we have
$\M^+,s\vDash\phi\iff \phi\in s,$ i.e. $\phi^{\M^+}=|\phi|$.
\end{lemma}

\begin{proof}
By induction on $\phi$. The nontrivial case is $\Delta\phi$.

Suppose that $\Delta\phi\in s$, to show that $\M^+,s\vDash\Delta\phi$. By supposition and axiom $\text{Dis}\Delta$, $\Delta(\phi\vee\psi)\vee\Delta(\neg\phi\vee\chi)\in s$ for arbitrary $\psi,\chi$. Thus $\Delta(\phi\vee\psi)\in s$ for all $\psi$, or $\Delta(\neg\phi\vee\chi)\in s$ for all $\chi$. The first case implies $|\phi|\in N^c(s)$, by induction hypothesis, $\phi^{\M^+}\in N^c(s)$, thus $\phi^{\M^+}\in N^+(s)$; similarly, the second case implies $(\neg\phi)^{\M^+}\in N^+(s)$. Both entails $\M^+,s\vDash\Delta\phi$.

Conversely, assume that $\M^+,s\vDash\Delta\phi$, then by induction hypothesis, $|\phi|\in N^+(s)$ or $|\neg\phi|\in N^+(s)$. If $|\phi|\in N^+(s)$, then $X\subseteq |\phi|$ for some $X\in N^c(s)$, and hence for some $\psi$, $|\psi|=X\in N^c(s)$, i.e. $\Delta(\psi\vee\psi')\in s$ for all $\psi'$. By $|\psi|\subseteq \phi$, we have $\vdash\psi\to\phi$, then $\vdash\psi\vee\phi\lra\phi$, by RE$\Delta$, $\vdash\Delta(\psi\vee\phi)\lra\Delta\phi$. Since $\Delta(\psi\vee\phi)\in s$, we conclude that $\Delta\phi\in s$. Similarly, from $|\neg\phi|\in N^+(s)$, it can follow that $\Delta\neg\phi\in s$, and therefore $\Delta\phi\in s$, as desired.
\end{proof}}

By Lemma~\ref{lemma.truthlemma-mc} and Lemma~\ref{lem.canonicalmodel}, we have

\begin{lemma}\label{lem.cm}
For all $\phi\in \mathcal{L}(\Delta)$, for all $s\in S^\Lambda$, we have
$(\M_0^\Lambda)^+,s\vDash\phi\iff \phi\in s,$ i.e. $\phi^{(\M_0^\Lambda)^+}=|\phi|$.
\end{lemma}

With a routine work, we obtain
\begin{theorem}\label{thm.comp-m}
${\bf M^\Delta}$ is strongly complete with respect to the class of $(m)$-frames.
\end{theorem}

We are now in a position to deal with the strong completeness of ${\bf R^\Delta}$. 
\weg{\begin{proposition}
$\Delta$C is valid on the class of quasi-filters.\footnote{It is worth remarking that in the case of $(i)$-frames, we need the property $(c)$ to provide the validity of $\Delta$C (see the proof of Prop.~\ref{thm.comp-EC}), by comparison, in the case of quasi-filters, we do not need $(c)$, since the validity of $\Delta$C is now guaranteed by $(i)$ and $(s)$ together.}
\end{proposition}

\begin{proof}
Let $\M=\lr{S,N,V}$ be a quasi-filter model and $s\in S$. Suppose that $\M,s\vDash\Delta\phi\land\Delta\psi$, then $\phi^\M\in N(s)$ or $(\neg\phi)^\M\in N(s)$, and $\psi^\M\in N(s)$ or $(\neg\psi)^\M\in N(s)$. Consider the following three cases:
\begin{itemize}
\item $\phi^\M\in N(s)$ and $\psi^\M\in N(s)$. By $(i)$, we obtain $\phi^\M\cap \psi^\M\in N(s)$, i.e. $(\phi\land\psi)^\M\in N(s)$, which gives $\M,s\vDash\Delta(\phi\land\psi)$.
\item $(\neg\phi)^\M\in N(s)$. By $(s)$, we infer $(\neg\phi)^\M\cup(\neg\psi)^\M\in N(s)$, i.e. $(\neg(\phi\land\psi))^\M\in N(s)$, which implies $\M,s\vDash\Delta(\phi\land\psi)$.
\item $(\neg\psi)^\M\in N(s)$. Similar to the second case, we can derive that $\M,s\vDash\Delta(\phi\land\psi)$.
\end{itemize}
\end{proof}}

\begin{proposition}\label{prop.i} 
For any system $\Lambda$ extending ${\bf R^\Delta}$, the minimal canonical model $\M^\Lambda_0$ has the property $(c)$. Hence, its supplementation $(\M^\Lambda_0)^+$ is an $(mc)$-model.
\end{proposition}

\begin{proof}
Suppose $X\in N^\Lambda_0(s)$ and $Y\in N^\Lambda_0(s)$, to show that $X\cap Y\in N^\Lambda_0(s)$. By supposition, there exist $\phi$ and $\chi$ such that $X=|\phi|$ and $Y=|\chi|$, and then $\Delta(\phi\vee\psi)\in s$ for every $\psi$, and $\Delta(\chi\vee\psi)\in s$ for every $\psi$. Using axiom $\Delta$C, we infer $\Delta((\phi\land\chi)\vee\psi)\in s$ for every $\psi$. Therefore, $|\phi\land\chi|\in N^\Lambda_0(s)$, i.e. $X\cap Y\in N^\Lambda_0(s)$. Thus $\M^\Lambda_0$ has the property $(c)$. Then it follows that $(\M^\Lambda_0)^+$ also possesses the property $(mc)$ from Prop.~\ref{prop.i&n}.
\end{proof}

The results below are now immediate, due to Lemma~\ref{lem.cm} and Prop.~\ref{prop.i}.
\begin{theorem}\label{thm.comp-r}
${\bf R^\Delta}$ is strongly complete with respect to the class of quasi-filters.
\end{theorem}

Now for ${\bf EMN^\Delta}$.
\begin{proposition}\label{prop.n-pres} For any system $\Lambda$ extending ${\bf EMN^\Delta}$, the minimal canonical model $\M^\Lambda_0$ has the property $(n)$. Hence, its supplementation $(\M^\Lambda_0)^+$ is an $(mn)$-model.
\end{proposition}

\begin{proof}
Let $s\in S^\Lambda$. By axiom $\Delta$N, we have $\Delta\top\in s$. This implies that $\Delta(\top\vee\psi)\in s$ for every $\psi$. Then $|\top|\in N^\Lambda_0(s)$, and thus $S\in N^\Lambda_0(s)$. We have now shown that $M^\Lambda_0$ has the property $(n)$. Then it follows that $(\M^\Lambda_0)^+$ is an $(mn)$-model from Prop.~\ref{prop.i&n}.
\end{proof}

By Prop.~\ref{prop.n-pres} and Lemma~\ref{lem.cm}, it follows immediately that
\begin{theorem}
${\bf EMN^\Delta}$ is strongly complete with respect to the class of $(mn)$-frames.
\end{theorem}

Finally, for ${\bf K^\Delta}$. For any system $\Lambda\supseteq {\bf K^\Delta}$, we have $\Lambda\supseteq {\bf R^\Delta}$ and $\Lambda\supseteq {\bf EMN^\Delta}$. By Prop.~\ref{prop.i}, $(\M^\Lambda_0)^+$ has the property $(c)$; by Prop.~\ref{prop.n-pres}, $(\M^\Lambda_0)^+$ has the property $(n)$; by the definition of $(N^\Lambda_0)^+$, $(\M^\Lambda_0)^+$ has the property $(m)$. Therefore, $(\M^\Lambda_0)^+$ is a filter. Then combining Lemma~\ref{lem.cm}, we conclude that
\begin{theorem}
${\bf K^\Delta}$ is strongly complete with respect to the class of filters.
\end{theorem}

\begin{remark}
It may be natural to ask whether the function $N^\Lambda$ in Def.~\ref{def.canonicalmodel} (or the minimal $N^\Lambda_0$) work for systems excluding $\Delta$M, e.g. ${\bf E^\Delta}$. We think the answer is negative, since the truth lemma (Lemma~\ref{lemma.truthlemma-mc}) indeed used the axiom $\Delta$M.
\end{remark}

We conclude this section with a diagram and some remarks.
By constructing countermodels, we can obtain the following cubes, which summarize the deductive powers of the 16 logics mentioned in this paper. Among these systems, ${\bf E^\Delta}={\bf EZ^\Delta}$ is the weakest system, and ${\bf KZ^\Delta}$ is the strongest system. An arrow from a system $S_1$ to another $S_2$ means that $S_2$ is deductively stronger than $S_1$. This is in line with the case of classical modal logics, cf. e.g.~\cite[Fig.~8.1]{Chellas:1980}.


\scalebox{1.0}{
\begin{tikzpicture}[z=0.35cm]
\node (000) at (0,0,0) {$\boxed{\blue{{\bf E^\Delta}={\bf EZ^\Delta}}}$};
\node (001) at (0,0,6) {$\boxed{\bf EC^\Delta}$};
\node (002) at (0,3,6) {$\boxed{\blue{{\bf ECZ^\Delta}}}$};
\node (010) at (0,6,0) {$\boxed{\blue{{\bf M^\Delta}(={\bf EM^\Delta})}}$};
\node (011) at (0,6,6) {$\boxed{\blue{{\bf R^\Delta}(={\bf EMC^\Delta})}}$};
\node (100) at (6,0,0) {$\boxed{\blue{{\bf EN^\Delta}={\bf ENZ^\Delta}}}$};
\node (101) at (6,0,6) {$\boxed{{\bf ECN^\Delta}}$};
\node (121) at (6,3,6) {$\boxed{\blue{{\bf ECNZ^\Delta}}}$};
\node (110) at (6,6,0) {$\boxed{\blue{{\bf EMN^\Delta}}}$};
\node (111) at (6,6,6) {$\boxed{\blue{{\bf K^\Delta}(={\bf EMCN^\Delta})}}$};
\node (333) at (6,10,6) {$\boxed{{\bf KZ^\Delta}(={\bf EMCNZ^\Delta})}$};
\node (313) at (0,10,6) {$\boxed{{\bf RZ^\Delta}=({\bf EMCZ^\Delta})}$};
\node (113) at (6,10,0) {$\boxed{{\bf EMNZ^\Delta}}$};
\node (131) at (0,10,0) {$\boxed{{\bf MZ^\Delta}(={\bf EMZ^\Delta})}$};
\draw[->] (010) -- (131);
\draw[->] (131) -- (113);
\draw[->] (110) -- (113);
\draw[->] (131) -- (313);
\draw[->] (313) -- (333);
\draw[->] (113) -- (333);
\draw[->] (111) -- (333);
\draw[dashed,->] (011) -- (313);

\draw[dashed,->] (001) -- (002);
\draw[->] (101) -- (121);
\draw[dashed,->] (002) -- (121);

\draw[->] (000) -- node[] {} (100);
\draw[dashed,->] (001) -- node[] {} (101);
\draw[dashed,->] (011) -- node[] {} (111);
\draw[->] (000) -- node[] {} (010);
\draw[dashed,->] (002) -- node[] {} (011);
\draw[->] (121) -- node[] {} (111);
\draw[dashed,->] (000) -- node[] {} (001);
\draw[dashed,->] (010) -- node[] {} (011);
\draw[->] (100) -- node[] {} (101);
\draw[->] (110) -- node[] {} (111);
\draw[->] (010) -- node[] {} (110);
\draw[->] (100) -- node[] {} (110);
\end{tikzpicture}
}

\begin{remark} Let Z denote $\Box\phi\to\Box\neg\phi$.
\begin{enumerate}
\item In this paper, among 16 logics in the above diagram, 10 logics were axiomatized (labeled with blue), and axiomatizations of the remaining 6 logics are open. Our conjecture is that if we replace $\Delta$M in the axiomatizations of logics on the second level with s$\Delta$M, namely $\Delta\phi\to\Delta(\phi\vee\psi)$, or equivalently $\Delta(\phi\land\psi)\to\Delta\phi$, then we will obtain the axiomatizations of logics on the third (topmost) level.

\item ${\bf S^\Delta}\subseteq {\bf SZ^\Delta}$ holds for any classical modal logics ${\bf S}$. This is because for any classes of frames $\mathbb{K}$ and $\mathbb{K}'$, if ${\bf S^\Delta}=\text{Th}_\Delta(\mathbb{K})$ and ${\bf SZ^\Delta}=\text{Th}_\Delta(\mathbb{K}')$, then $\mathbb{K}'\subseteq \mathbb{K}$. For instance, although we do not know yet what the axiomatizations of logics ${\bf EC^\Delta}$ and ${\bf ECN^\Delta}$ are, we do know ${\bf EC^\Delta}\subseteq {\bf ECZ^\Delta}$ and ${\bf ECN^\Delta}\subseteq {\bf ECNZ^\Delta}$, since $\mathbb{F}_{cz}\subseteq \mathbb{F}_{c}$ and $\mathbb{F}_{cnz}\subseteq \mathbb{F}_{cn}$.

\item However, the equation ${\bf S^\Delta}={\bf SZ^\Delta}$ does not hold for any classical modal logics ${\bf S}$, but only holds for the cases when ${\bf S}={\bf E}$ or ${\bf S}={\bf EN}$ (in comparison, ${\bf E}\neq {\bf EZ}$ and ${\bf EN}\neq {\bf ENZ}$ as is easily verified). This may be explained with help of a notion of `$c$-variation' introduced in~\cite[Def.~8]{Fan:2018a}, which we call now `complementation', for the sake of reference.

Given a neighborhood function $N$, we define its {\em complementation} $N^z$ as follows:
$$N^z(s)=\{X\subseteq S\mid X\in N(s)\text{ or }S\backslash X\in N(s)\}.$$

Given $\mathcal{F}=\lr{S,N}$ and $\M=\lr{S,N,V}$, put $\mathcal{F}^z=\lr{S,N^z}$ and $\M^z=\lr{S,N^z,V}$. Then one may easily verify that for every $\phi\in\mathcal{L}(\Delta)$, for every $s\in S$, $\M,s\vDash\phi$ iff $\M^z,s\vDash\phi$. Consequently, $\text{Th}_\Delta(\mathcal{F})= \text{Th}_\Delta(\mathcal{F}^z)$. Therefore, if $\mathbb{K}$ is a class of frames and we denote $\mathbb{K}^z=\{\mathcal{F}^z\mid \mathcal{F}\in\mathbb{K}\}$, then $\text{Th}_\Delta(\mathbb{K})=\text{Th}_\Delta(\mathbb{K}^z)$, and $\mathbb{K}\subseteq \mathbb{K}^z$. Note that the complementation of a frame is a $(z)$-frame: $(\mathbb{F}_{\text{all}})^z=\mathbb{F}_{z}=\mathbb{F}_{\text{all}}$ (the $\subseteq$ is obvious, for the $\supseteq$ part, notice that $\mathbb{F}_{\text{all}}\subseteq (\mathbb{F}_{\text{all}})^z$), and the complementation of an $(n)$-frame is an $(nz)$-frame: $(\mathbb{F}_{n})^z=\mathbb{F}_{nz}=\mathbb{F}_n$. Therefore, ${\bf E^\Delta}={\bf EZ^\Delta}$ (recall that ${\bf E^\Delta}=\text{Th}_\Delta(\mathbb{F}_\text{all})$ and ${\bf EZ^\Delta}=\text{Th}_\Delta(\mathbb{F}_z)$, see Thm.~\ref{thm.comp-E}; and $\mathbb{F}_\text{all}=\mathbb{F}_z$), and similarly ${\bf EN^\Delta}={\bf ENZ^\Delta}$. This explains why at the bottom of the cube of $\Delta$-theories we did not have the difference between the logics with and without Z. But the complementation of a $(c)$-frame is not necessarily a $(c)$-frame,\footnote{Consider a frame $\mathcal{F}=\lr{S,N}$ where $S=\{s,t\}$ and $N(s)=\{\{s\}\}$ and $N(t)=\emptyset$. It can be easily seen that $\mathcal{F}$ is a $(c)$-frame. However, its complementation is not: on one hand, from $\{s\}\in N(s)$ it follows that $\{s\}\in N^z(s)$ and $\{t\}\in N^z(s)$; on the other hand, because $\emptyset\notin N(s)$ and $S\notin N(s)$, we have $\emptyset=\{s\}\cap\{t\}\notin N^z(s)$.} hence not always a $(cz)$-frame. In fact, as shown in items (iii) and (v) of Prop.~\ref{prop.validinvalid}, the class of $(c)$-frames and the class of $(cz)$-frames can be distinguished by $\Delta$C, that is, $\Delta\text{C}\in\text{Th}_\Delta(\mathbb{F}_{cz})$ but $\Delta\text{C}\notin \text{Th}_\Delta(\mathbb{F}_c)$, and thus ${\bf EC^\Delta}\neq{\bf ECZ^\Delta}$. Similarly, ${\bf ECN^\Delta}\neq {\bf ECNZ^\Delta}$.

\item Although we have no idea yet about what the axiomatizations of logics on the third level exactly are, we do know that the logics on that level differ from the logics on the second level. Indeed, the smallest logic on the third level, ${\bf MZ^\Delta}$, contains s$\Delta$M (since $\text{s}\Delta\text{M}\in \text{Th}_\Delta(\mathbb{F}_{mz})$ by item (xi) of Prop.~\ref{prop.validinvalid}), which does not belong to the greatest logic on the second level, ${\bf K^\Delta}$ (because $\text{s}\Delta\text{M}\notin \text{Th}_\Delta(\mathbb{F}_{mcn})$ by item (vii) of Prop.~\ref{prop.validinvalid}). 


\item  
${\bf ECZ^\Delta}\subseteq {\bf R^\Delta}$ and ${\bf ECNZ^\Delta}\subseteq {\bf K^\Delta}$, even though the corresponding inclusions of $\Box$-logics do not hold (because $\Box\phi\to\Box\neg\phi$ is not provable in ${\bf K}$ and thus in ${\bf R}$). This can be easily seen from the axiomatizations:

    ${\bf ECZ^\Delta}={\bf E^\Delta}+\Delta\text{C}$ and ${\bf R^\Delta}={\bf E^\Delta}+\Delta\text{C}+\Delta\text{M}$;

    ${\bf ECNZ^\Delta}={\bf E^\Delta}+\Delta\text{C}+\Delta\text{N}$ and ${\bf K^\Delta}={\bf E^\Delta}+\Delta\text{C}+\Delta\text{N}+\Delta\text{M}$.

\item It may be interesting to compare ${\bf K^\Delta}$ with Kuhn's minimal non-contingency logic ${\bf K}\Delta$~\cite{DBLP:journals/ndjfl/Kuhn95}.\footnote{Kuhn miswrote ${\bf K}\Delta$ on p. 231 as ${\bf K4}\Delta$. This was pointed out by George Schumm in the review of \cite{DBLP:journals/ndjfl/Kuhn95}, see~\cite{Kuhn:1996}.} It turns out that the two logics are equivalent, since $\Delta\phi\lra\Delta\neg\phi$ (our axiom $\Delta$Equ) is provable in ${\bf K}\Delta$, by using $\Delta\neg\phi\to\Delta\phi$ (denoted {\bf A1} there) and classical propositional calculus (denoted {\bf PL} there and TAUT here) and replacement of equivalents for $\Delta$ (denoted {\bf RE} there and RE$\Delta$ here), and $\Delta\top$ (our axiom $\Delta$N) is interderivable with Kuhn's inference rule $\dfrac{\phi}{\Delta\phi}$ (denoted ${\bf R\Delta}$ there) due to replacement of equivalents for $\Delta$.

\item The axiomatization of the logic of the form ${\bf LZ^\Delta}$ can be obtained by simply taking the axiomatization of ${\bf L}$ and replacing everywhere $\Box$ with $\Delta$. In particular, we obtain $\Delta\phi\lra\Delta\neg\phi$. This is because any logic with Z has $\Box \phi\lra(\Box\phi\vee\Box\neg\phi)$ ($\Box\neg\phi\to\Box\phi$ can be derived from Z by using classical propositional calculus and replacement of equivalents for $\Box$, RE), that is, $\Box\phi\lra \Delta\phi$. Semantically, this can be explained by the fact that $\mathbb{F}_z\vDash\Box\phi\lra\Delta\phi$.

\end{enumerate}
\end{remark}

\section{Reflection: how does the function $\lambda$ arise?}\label{sec.reflection-lambda}

As noted, in order to show the completeness of proof systems including $\Delta$M, a crucial part is to define a suitable canonical function, i.e. $N^\Lambda$, which is inspired by the function $\lambda$ in~\cite{DBLP:journals/ndjfl/Kuhn95}. The $\lambda$ is very important for the definition of canonical relation and thus for the completeness proof in the cited paper. It is this function that helps find simple axiomatization for the minimal contingency logic under Kripke semantics, so to speak, since we can obtain his axiomatization from the truth lemma proof, whichever axioms or rules needed are added. Despite its importance, the author did not say any intuitive idea about $\lambda$. And this function was thought of as `ingenious' creation by some other researchers, say Humberstone~\cite[p.~118]{Humberstone:2002}\footnote{Although~\cite[p.~1279]{Humberstone:2013} states ``This simplifies a definition from Humberstone~\cite{Humberstone95}, p. 221f.'', this wording leaves it open whether Kuhn's $\lambda$ is the same function as Humberstone's $\lambda$, or a different one. In fact, by personal communication, Humberstone has confirmed that he had not explicitly considered this question before seeing a draft of of the current paper settling it in favor of the first alternative --- that what is simplified is {\em the way} the function is picked out rather than {\em which function} is picked out.} and Fan, Wang and van Ditmarch~\cite[p.~101]{Fanetal:2015}. But how does the function arise? In this section, we unfold the mystery of $\lambda$, and show that it is actually equal to a related function $\lambda$ proposed in Humberstone~\cite{Humberstone95}.

To show completeness of minimal contingency logic under Kripke semantics, Humberstone~\cite[p.~219]{Humberstone95} defined the canonical relation $R^c$ as $xR^cy$ iff $\lambda(x)\subseteq y$, where, denoted by H's $\lambda$,
$$\lambda(x)=\{\phi\mid \Delta \phi\in x\text{ and }\forall \psi\text{ such that }\vdash \phi\to \psi,\Delta \psi\in x\}.\footnote{In the definition of $\lambda$, Humberstone used $A$ and $B$ rather than $\phi$ and $\psi$, respectively. To maintain the consistency of notation in this paper, we here use $\phi$ and $\psi$ instead.}$$
The reason for defining the function $\lambda$ in such a way, is that the author would like to `simulate' the canonical relation of the minimal modal logic, which is defined via $xRy$ iff $\lambda(x)\subseteq y$, where $\lambda(x)=\{\phi\mid \Box \phi\in x\}$. This can be seen from several passages:

\medskip

{\em The intuitive idea is that for $x\in W$, $\lambda(x)$ is the set of formulas which are necessary at $x$. We think of $\lambda(x)$ as a ``labeling'' of all formulas $A$ such that $\Delta A\in x$, labeling each such formula as {\em Necessary} (recorded by putting $A$ into $\lambda(x)$) or else as {\em Impossible} (putting $\neg A$ into $\lambda(x)$).}

$\cdots$

{\em The idea of the entry condition on $A$, that only such $A$ (with $\Delta A\in x$) should be labeled as Necessary if all their consequences are non-contingent, is that $\cdots$, those non-contingencies which qualify as such because they, rather than their negations, are necessary and have only non-contingent consequences, since those consequences are themselves necessary.}~\cite[p.~219]{Humberstone95}

\medskip

Then the function $\lambda$ was simplified, and accordingly, the completeness proof was simplified in~\cite{DBLP:journals/ndjfl/Kuhn95}. There, $\lambda(x)$, denoted by K's $\lambda$, is defined as:
$$\lambda(x)=\{\phi\mid\forall \psi, \Delta(\phi\vee \psi)\in x\}.$$
In the sequel, we will demonstrate that, in fact, K's $\lambda$ is equal to H's $\lambda$.

To begin with, notice that $\vdash \phi\to \phi$, thus the part following `and' in the H's $\lambda$ definition entails $\Delta \phi\in x$. Therefore, the H's $\lambda(x)$ is equal to a simplified version:
$$\lambda(x)=\{\phi\mid \forall \psi\text{ such that }\vdash \phi\to \psi,\Delta \psi\in x\}.$$
Then it is sufficient to show that the simplified $\lambda$ is further equal to K's $\lambda$, even in the setting of arbitrary neighborhood contingency logics (as opposed to Kripke contingency logics).
\begin{proposition} Let $x$ be a maximal consistent set. Given the rule RE$\Delta$, the following statements are equivalent.\footnote{RE$\Delta$ is just ($\Delta$Cong) in~\cite{Humberstone95}.}

(1) For every $\psi\text{ such that }\vdash \phi\to \psi$, $\Delta \psi\in x$.

(2) For every $\psi$, $\Delta(\phi\vee \psi)\in x$.
\end{proposition}

\begin{proof}
$(1)\Longrightarrow(2)$: suppose (1) holds. Since $\vdash\phi\to\phi\vee\psi$, then it is immediate by (1) that $\Delta(\phi\vee\psi)\in x$, namely (2).

$(2)\Longrightarrow(1)$: suppose (2) holds, to show (1). For this, assume that $\vdash\phi\to\psi$, then $\vdash\phi\vee\psi\lra\psi$, by RE$\Delta$, $\vdash\Delta(\phi\vee\psi)\lra\Delta\psi$. By (2), we obtain that $\Delta\psi\in x$, as desired.
\end{proof}

\section{Concluding Discussions}\label{sec.concl}

In this paper, by defining suitable neighborhood canonical functions, we presented a family of contingency logics under neighborhood semantics. In particular, inspired by Kuhn's function $\lambda$ in~\cite{DBLP:journals/ndjfl/Kuhn95}, we defined a desired canonical neighborhood function, and then axiomatized monotone contingency logic and regular contingency logic and other logics including the axiom $\Delta$M, thereby answering two open questions raised in~\cite{Bakhtiarietal:2017}. We then reflected on the function $\lambda$, and showed that it is actually equal to Humberstone's function $\lambda$ in~\cite{Humberstone95}, even in the setting of arbitrary neighborhood contingency logics.

Moreover, as we show in Appendix, in ${\bf M^\Delta}$, $\Delta$M can be replaced by $\Delta\phi\to\Delta(\phi\to\psi)\vee \Delta(\neg\phi\to\chi)$, and in ${\bf R^\Delta}$, $\Delta$C can be replaced by $\Delta(\psi\to\phi)\land\Delta(\neg\psi\to\phi)\to\Delta\phi$.\weg{\footnote{The hard part is the direction from $\Delta(\psi\to\phi)\land\Delta(\neg\psi\to\phi)\to\Delta\phi$ to $\Delta$C. The proof details for this, we refer to~\cite[Prop.~50]{DBLP:journals/corr/FanWD13}, where knowing whether operator $\textit{Kw}$ is the epistemic reading of $\Delta$.}} Thus we can also adopt these two alternative formulas to axiomatize monotone contingency logic and regular contingency logic. Therefore, it was {\em wrong} to claim that ``This raises the questions of what the axiomatizations are of monotone contingency logic and regular contingency logic. $\cdots$ one {\em cannot} fill these gaps with the axioms $\Delta\phi\to\Delta(\phi\to\psi)\vee \Delta(\neg\phi\to\chi)$ and $\Delta(\psi\to\phi)\land\Delta(\neg\psi\to\phi)\to\Delta\phi$. So these
questions remain open.'' on~\cite[p.~62]{Bakhtiarietal:2017} and~\cite[pp.~124--125]{Bakhtiarinoodeh:2017thesis}. Also, we answer the two open questions therein.

Recall that an `almost definability' schema, $\nabla\chi\to(\Box\phi\lra(\Delta\phi\land\Delta(\chi\to\phi)))$, is proposed in~\cite{Fanetal:2014}, and shown in~\cite{Fanetal:2015} to be applied to axiomatize contingency logic over much more Kripke frame classes than Kuhn's function $\lambda$ and other variations. Therefore, it may be natural to ask if the schema can also work in the neighborhood setting. The canonical neighborhood function inspired by the schema seems to be
$$N(s)=\{|\phi|\mid \Delta\phi\land\Delta(\psi\to\phi)\in s\text{ for some }\nabla\psi\in s\}.$$
Unfortunately the answer seems to be negative. The reason can be explained as follows. Although $N^\Lambda$ in Def.~\ref{def.canonicalmodel} is almost monotonic in the sense that if $|\phi|\in N^\Lambda(s)$ and $|\phi|\subseteq |\psi|$, then $|\psi|\in N^\Lambda(s)$, as can be easily seen from the proof of Lemma~\ref{lem.canonicalmodel}, in contrast, as one may easily verify, $N$ is not almost monotonic in the above sense, i.e., it fails that if $|\phi|\in N(s)$ and $|\phi|\subseteq |\psi|$, then $|\psi|\in N(s)$.
This can also explain why $N^\Lambda$ works well for systems extending ${\bf M^\Delta}$\weg{monotone and regular contingency logics and other logics including the axiom $\Delta$M}. Despite this fact, this $N^\Lambda$ does not apply to systems excluding the axiom $\Delta$M, since we need this axiom to ensure the truth lemma (Lemma~\ref{lemma.truthlemma-mc}). It is also worth noting that this $N^\Lambda$ is smaller than that in the case of classical contingency logic (Def.~\ref{def.cm-no-m}), thus we cannot address all neighborhood contingency logics in a unified way.\footnote{In contrast, the canonical neighborhood function used in the completeness proof of classical modal logic is the smallest neighborhood function among canonical neighborhood functions used in the completeness proofs of all neighborhood modal logics. Cf. e.g.~\cite{Chellas:1980}.} This indicates that the completeness proofs of these logics are nontrivial. Besides, $N^\Lambda$ seems not workable for proper extensions of ${\bf K^\Delta}$, which we leave for future work.

\section{Acknowledgements}

This research is supported by the project 17CZX053 of National Social Science Fundation of China. We would like to thank Lloyd Humberstone and two anonymous referees for careful reading of earlier versions and making insightful comments, which help improve the paper substantially.

\bibliographystyle{plain}
\bibliography{biblio2017,biblio2016}

\section*{Appendix}

Recall that $\Delta$M stands for $\Delta\phi\to\Delta(\phi\vee\psi)\vee\Delta(\neg\phi\vee\chi)$, and $\Delta\text{C}$ stands for $\Delta\phi\land\Delta\psi\to\Delta(\phi\land\psi)$. And in the Conclusion section (Sec.~\ref{sec.concl}), we claim that ``in ${\bf M^\Delta}$, $\Delta$M can be replaced by $\Delta\phi\to\Delta(\phi\to\psi)\vee \Delta(\neg\phi\to\chi)$ (denoted by $\Delta\text{M}'$), and in ${\bf R^\Delta}$, $\Delta$C can be replaced by $\Delta(\psi\to\phi)\land\Delta(\neg\psi\to\phi)\to\Delta\phi$ (denoted by $\Delta\text{C}'$)''. In this appendix, we verify this claim, that is,
$$(1)~~~~~~{\bf E^\Delta}+\Delta\text{M}={\bf E^\Delta}+\Delta\text{M}',$$
and
$$(2)~~~~~~{\bf M^\Delta}+\Delta\text{C}={\bf M^\Delta}+\Delta\text{C}'.$$

It is straightforward to show $(1)$, by using only $\text{TAUT}$, $\Delta\text{Equ}$, $\text{MP}$ and $\text{RE}\Delta$. For $(2)$, the `$\supseteq$' part, that is, $\Delta\text{C}'$ is provable in ${\bf M^\Delta}+\Delta\text{C}$, is easy, shown as follows.
\[
\begin{array}{lll}
(i)& \Delta(\psi\to\phi)\land\Delta(\neg\psi\to\phi)\to\Delta((\psi\to\phi)\land(\neg\psi\to\phi))& \Delta\text{C}\\
(ii)& (\psi\to\phi)\land(\neg\psi\to\phi)\lra\phi&\text{TAUT}\\
(iii)&\Delta((\psi\to\phi)\land(\neg\psi\to\phi))\lra\Delta\phi&(ii),\text{RE}\Delta\\
(iv)&\Delta(\psi\to\phi)\land\Delta(\neg\psi\to\phi)\to\Delta\phi&(i),(iii),\text{MP}\\
\end{array}
\]

The `$\subseteq$' part, that is, $\Delta\text{C}$ is provable in ${\bf M^\Delta}+\Delta\text{C}'$, is harder.
\[
\begin{array}{lll}
(i)&\Delta\phi\to\Delta(\phi\vee(\phi\land\psi))\lor\Delta(\neg\phi\vee\neg\psi)&\Delta\text{M}\\
(ii)&\Delta\psi\to\Delta(\psi\vee\neg\phi)\vee \Delta(\neg\psi\vee\neg\phi)&\Delta\text{M}\\
(iii)&\Delta(\neg\phi\vee\neg\psi)\lra\Delta(\neg\psi\vee\neg\phi)&\text{TAUT},\text{RE}\Delta\\
(iv)&\Delta\phi\land\Delta\psi\to\Delta(\neg\phi\vee\neg\psi)\vee(\Delta(\phi\vee(\phi\land\psi))\land\Delta(\psi\vee\neg\phi))&(i)-(iii)\\
(v)&\Delta(\psi\vee\neg\phi)\lra\Delta(\neg\phi\vee(\phi\land\psi))&\text{TAUT},\text{RE}\Delta\\
(vi)&\Delta\phi\land\Delta\psi\to\Delta(\neg\phi\vee\neg\psi)\vee(\Delta(\phi\vee(\phi\land\psi))\land\Delta(\neg\phi\vee(\phi\land\psi))&(iv),(v)\\
(vii)&\Delta(\phi\vee(\phi\land\psi))\lra\Delta(\neg\phi\to(\phi\land\psi))&\text{TAUT},\text{RE}\Delta\\
(viii)&\Delta(\neg\phi\vee(\phi\land\psi)\lra\Delta(\phi\to(\phi\land\psi))&\text{TAUT},\text{RE}\Delta\\
(ix)&\Delta(\neg\phi\vee\neg\psi)\lra\Delta(\phi\land\psi)&\Delta\text{Equ}\\
(x)&\Delta\phi\land\Delta\psi\to\Delta(\phi\land\psi)\vee(\Delta(\neg\phi\to(\phi\land\psi))\land \Delta(\phi\to(\phi\land\psi)))&(vi)-(ix)\\
(xi)&\Delta(\neg\phi\to(\phi\land\psi))\land \Delta(\phi\to(\phi\land\psi))\to \Delta(\phi\land\psi)&\Delta\text{C}'\\
(xii)&\Delta\phi\land\Delta\psi\to\Delta(\phi\land\psi)&(x),(xi)\\
\end{array}
\]

\end{document}